\documentclass[11pt, a4paper,oneside]{amsart}
\usepackage[T1]{fontenc}
\usepackage{carlocommands}

\title{Non-contractive logics, paradoxes, and multiplicative quantifiers}

\author[]{Carlo Nicolai}
\address{Department of Philosophy, King’s College, London, United Kingdom}
\email[]{carlo.nicolai@kcl.ac.uk}
\thanks{}

\author[]{Mario Piazza}
\address{Scuola Normale Superiore di Pisa, Classe di Lettere e Filosofia, Pisa, Italy}
\email[]{mario.piazza@sns.it}
\thanks{}

\author[]{Matteo Tesi$^1$}
\address{Scuola Normale Superiore di Pisa, Classe di Lettere e Filosofia, Pisa, Italy}
\email[]{matteo.tesi@sns.it}
\thanks{$^1$Corresponding author}

\date{September 2022}

\begin{document}
\maketitle
\begin{abstract}  The paper investigates from a proof-theoretic perspective various non-contractive logical systems circumventing logical and semantic paradoxes. {Until recently, such systems only displayed additive quantifiers  (Gri\v{s}in, Cantini). Systems with multiplicative quantifers have also been proposed in the 2010s (Zardini), but they turned out to be inconsistent with the naive rules for truth or comprehension.}
We start by presenting a first-order system for disquotational truth with additive quantifiers and we compare it with Gri\v{s}in set theory. 
We then analyze {the reasons behind the inconsistency phenomenon affecting multiplicative quantifers}: after interpreting the exponentials in affine logic as vacuous quantifiers, we show how such a logic can be simulated within a  truth-free
fragment of a system with multiplicative quantifiers. Finally, we prove  that the logic of these multiplicative quantifiers {(but without disquotational truth)} is consistent, by showing that an infinitary version of the cut rule can be eliminated. This paves the way to a syntactic approach to the proof theory of infinitary logic with infinite sequents.

\end{abstract}
\section{Introduction}




Since \cite{fitch1936} it is well-known that the contraction rule plays an essential role in the derivation of logical and semantic paradoxes such as the Liar, Russell's and Curry's.  In the last decades there has been a renewed interest in non-contractive logical systems  -- as Fitch called them --  that block these  paradoxes by dropping contraction from their sequent calculus formulation. 
\cite{gri82} established the consistency of a ``set theory'' -- or better, a property theory -- based on what is nowadays called affine logic (i.e. linear logic equipped with the weakening rule) extended with na\"ive comprehension. \cite{petersen2000} further elaborates on Gri\v{s}in's proposal by giving a proof-theoretic analysis of a system with unrestricted abstraction and some additional axioms. Moreover, \cite{can03} embeds combinatory logic in Gri\v{s}in set theory, thereby establishing its undecidability. As we shall see below, it's not difficult to see that Gri\v{s}in's set theory gives rise to a consistent theory of disquotational truth. 

Nevertheless, it also clear that the solution thus provided cannot be the whole story since it only features {\em additive} quantifiers, which are in effect {\em classical} quantifiers in disguise. Indeed, as stressed in \cite{mon04} and \cite{paoli2005}, the difference between the additive universal quantifier and the multiplicative one may be roughly understood as the one between {\it any} and {\it every}. 
Given the splitting phenomenon determined by the absence of contraction, the additive quantifiers generalize additive connectives, but there is no logical corresponding device generalizing multiplicative ones \cite{bla92,mon04,mapa14}. A spontaneous  way of conceiving of multiplicative quantifiers is to identify the universal and the existential quantifiers with {\em infinitary} multiplicative conjunctions and disjunctions, respectively

\begin{center}

$\forall xA \equiv A(x/t_1) \otimes A(x/t_2) \otimes \dots$\\
$\exists xA \equiv A(x/t1) \parr A(x/t_2) \parr \dots$
\end{center}

Following this intuition, \cite{zar11} presented a theory of disquotational truth based on a purely multiplicative fragment of affine logic featuring infinitary quantifiers. 
However, such a theory has received enough attention to make it clear that: (i) it cannot be extended with suitable primitive recursive functions \cite{daro18}; 
 (ii) the attempted proof of consistency of the system via cut-elimination  contains a gap \cite{fje20}; (iii) the system is outright inconsistent given some plausible principles for vacuous quantification \cite{fjol21}.

 In the paper we  contribute to the understanding of the non-contractive landscape, by clarifying a cluster of intertwined issues. In particular:
 \smallskip
 
 \begin{itemize}
    \item We simplify the cut-elimination proof for Gri\v{s}in's set theory carried in \cite{can03}, while fixing a problem in Cantini's strategy. We also show that the (apparently weaker) theory of disquotational truth based on affine logic supports Cantini's derivation of L\"ob's principle given a $\mathtt{K4}$ modality. 
    \smallskip
    \item We show that the very rules for vacuous quantification that are responsible for the inconsistency of Zardini's system can be employed to recover full classical logic in the context of affine logic. In fact, we show that there exists an \emph{exact} translation of (predicate, infinitary) classical logic into affine logic with vacuous quantification. 
        \smallskip
    \item In the field of linear logic \cite{Girard87}, the  dismissed contraction and weakening can be recovered and thus controlled using exponentials: !, ?, which in essence behave as $\mathtt{S4}$ modalities. We provide a new perspective on exponentials by interpreting them as vacuous quantifiers. In particular, we show how to simulate affine linear logic within a proper fragment of the system of multiplicative quantifiers, by giving a sound and faithful translation.
        \smallskip
        \item We directly show that Zardini’s cut-elimination algorithm
         is based on a proof-manipulation that does not preserve provability. 
            \smallskip
    \item Finally, we show 
    that an infinitary version of the cut rule can be eliminated from the purely logical system featuring infinitary quantifiers. 
 \end{itemize}
 \smallskip
The last point also answers to a question recently posed by \cite{petersen2022}. However, such result has a proof-theoretic interest {\em per se}, beyond a non-contractive approach to paradoxes. The proof theory of well-founded infinitely branching derivations has been extensively studied and has found large application in the context of ordinal analysis. Well-founded infinitary derivations involving sequents with infinitely many formulas have received less attention. The investigations concerning this kind of calculi have been conducted using semantical methods (see \cite{takProof}). 

A semantic argument can be employed to show the cut-free completeness of a calculus for infinitary classical logic with infinite sequents. However, this strategy works only insofar as a semantic presentation is available and a syntactic procedure for cut-elimination is currently lacking. We provide a syntactic proof of cut-elimination for the system involving sequents with infinitely many formulas for the logic of multiplicative quantifiers. 

The plan for the paper is as follows. Section 2 discusses a contraction-free and cut-free system for disquotational truth in relation to Grishin set theory. Section 3 shows how the exponentials ! and ? can be demodalized by conceiving them in terms of vacuous quantifiers within (a truth-free fragment of) Zardini’s system. Section 4 splits into a {\em pars destruens} -- which investigates the reasons leading to the inconsistency of Zardini’s system -- and a {\em pars construens} which yields a cut-elimination procedure for multiplicative quantifiers. Section 5 concludes with sketching some open problems triggered by our results.
\section{Contraction and the paradoxes}

Non-contractive approaches to the logical and semantic paradoxes are known to be formally successful. Without contraction, it's possible to extend a standard cut-elimination procedure for first-order affine logic without exponentials (henceforth, affine logic $\al$) to its extension with na\"ive rules for truth, (class-)membership, predication. 
\begin{dfn}[Affine Logic $\al$] $\Gamma,\Delta,\Theta,\Lambda...$ range over finite multisets of formulae of a countable, first-order Tait language $\mc{L}$.\footnote{For the definition of the Tait language, see \cite{sch77}.}
    \begin{align*}
         &\ax{}\Rlb{{\sc in}}
         \uinf{\vdash\Gamma,P,\ovl{P}}
            \DisplayProof   \\[10pt]
            &\ax{\vdash\Gamma,A_i }\Rlb{$A_i$,\text{\scriptsize{$i=1,2$}}}
                \uinf{\vdash\Gamma,A_1\oplus A_2}
                    \disp
                &&\ax{\vdash\Gamma,A}
                    \ax{\vdash\Gamma,B}\Rlb{$\with$}
                        \binf{\vdash\Gamma,A\with B}
                            \disp\\[10pt]
            &\ax{\vdash\Gamma,A,B}\Rlb{$\parr$}
                    \uinf{\vdash\Gamma,A\parr B}
                        \disp
                &&\ax{\vdash\Gamma,A}
                    \ax{\vdash\Delta,B}\Rlb{$\otimes$}
                        \binf{\vdash\Gamma,\Delta,A\otimes B}
                            \disp\\[10pt]
            &\ax{\vdash\Gamma,A(y/x)}\Rlb{$\forall$}
                \uinf{\vdash\Gamma,\forall x A}
                    \disp
                &&\ax{\vdash\Gamma,A(t/x)}\Rlb{$\exists$}
                    \uinf{\vdash\Gamma,\exists x A}
                        \disp
    \end{align*}
\end{dfn}
\noindent \emph{Linear logic} without exponentials is obtained from affine logic by focusing on initial sequents of form $\vdash P,\ovl{P}$.

Let $\mc{L}^+$ be a language featuring:
\begin{itemize}
    \item For $n,m\in \nat$, $n$-ary predicates $\sat{n}{m}$ and their dual;
    \item The logical symbols of $\al$;
    \item The $\lambda$ term forming operator $\lambda \cdot.\cdot$;
    \item Variables $v_1,v_2,\ldots$ (we employ $x,y,z$ for metavariables). 
\end{itemize}
For formulae $A\in \mc{L}^+$, $\lambda x A$ is a term whose free variables are the free variables of $A$ minus $x$. We abbreviate:
\[
    \lambda x_1\ldots x_n\,A:= \lambda x_1.(\ldots \lambda x_n \, A \ldots).
\]
Notice that we allow for ``self-referential'' names to be built in the system. For instance, we allow for the existence of terms $l$  such that
\[
    l:= \corn{\ovl{\sat{1}{0}}(l)}. 
\]
The term $l$, as we shall see shortly, plays the role of a name for a Liar sentence. Similar terms are available for other paradoxical sentences such as Russell's, Curry's, and so on. 

\begin{dfn}[Semantic Extensions of $\al$]\hfill\label{dfn:uts}
\begin{enumeratei}
    \item The system $\uts{n}{m}$ is obtained by formulating $\al$ in $\mc{L}^+$ and by adding the rules
    \begin{align*}
        & 
        \ax{\vdash\Gamma, A(t_1,\ldots,t_n)}\Rlb{$\sat{n}{m}$}
             \uinf{\vdash\Gamma, \sat{n}{m}(\lambda x_1\ldots x_n\,A,t_1\ldots t_n)}
                \disp\\[1ex]
        &
        \ax{\vdash\Gamma, \ovl{A}(t_1,\ldots,t_n)}\Rlb{$\ovl{\sat{n}{m}}$}
             \uinf{\vdash\Gamma, \ovl{\sat{n}{m}}(\lambda x_1\ldots x_n\,A,t_1\ldots t_n)}
                \disp
    \end{align*}
    for all formulae $A$ with exactly $m$ free variables. 
    \item $\mathtt{UTS}$ comprises rules for $\sat{n}{m}$ for all $n,m\in \nat$. 
\end{enumeratei}
\end{dfn}
\begin{remark}
    The template provided by the theories $\uts{n}{m}$ enable us to define several systems that are relevant for the analysis of the paradoxes in a non-contractive setting. As we shall see shortly, the systems $\uts{1}{m}$, for each $m\in \nat$, correspond to Gri\^{s}in set theory (which, being non-extensional, is perhaps better categorized as a property theory). A non-contractive theory of \emph{disquotational truth} corresponds to $\uts{1}{0}$. 
\end{remark}

Derivations in $\al$ and extensions thereof are finite trees that are locally correct with respect to the rules just given. Cantini in \cite{can03} provides a  cut-elimination strategy for the system $\uts{1}{m}$. The strategy relies on a triple induction on, respectively, the number of na\"ive comprehension rules, the grade of the cut-formula, and the level of the cut. The strategy, as it stands, cannot deal satisfactorily with some of the cases, for instance the one in which the last inference in one of the branches before a cut is an additive conjunction and in which the cut formula is not principal in the last inference.\footnote{The triple induction may be repairable -- as suggested by Cantini in personal communication -- by redefining what Cantini calls $\in$-complexity for additive rules, by taking in particular the maximum of the $\in$-complexity of the premisses instead of their sum.  } We circumvent the problem by showing that an induction on a \emph{single parameter} suffices. In order to do this, we provide a slightly nonstandard measure of length of the derivation. 
\begin{dfn}\label{dfn:height}
    Given a proof $\pi$, its height $h(\pi)$ is given by the following recursion:
    \begin{itemize}\setlength\itemsep{1ex}
        \item $h(\pi)=1$ for $\pi$ an instance of ({\sc in});
        \item $h(\pi)=\mathrm{max}(h(\pi_0),h(\pi_1))+1$, with $\pi$ ending with an application of $(\with)$ to $\pi_0$ and $\pi_1$; 
        \item $h(\pi)=h(\pi_0)+h(\pi_1)$, with $\pi$ ending with an application of $(\otimes)$ to $\pi_0$ and $\pi_1$; 
        \item $h(\pi)=h(\pi_0)+1$ in all other cases. 
    \end{itemize}
\end{dfn}

\begin{prop}
   Cut is admissible in $\mathtt{UTS}$. Therefore, $\mathtt{UTS}$ is consistent. 
\end{prop}
\begin{proof}
    The proof rests on the following reduction lemma:
    \begin{itemize}\label{lem:red}
        \item[\textsc{(r)}] {if $\mc{D}_0$ and $\mc{D}_0$ are cut-free proofs of $\Gamma,A$ and $\Delta,\ovl{A}$, respectively, then there is a cut-free proof $\mc{D}$ of $\Gamma,\Delta$ with $h(\mc{D})\leq h(\mc{D}_0)+h(\mc{D}_1)$.  }
    \end{itemize}
    \textsc{(r)} is proved by an induction on $h(\mc{D}_0)+h(\mc{D}_1)$. We consider two cases for illustration. If the ``cut formulae'' are principal in the last inference, and they are obtained by  (for notational simplicity) $\sat{n}{n}$ and $\ovl{\sat{n}{n}}$, respectively, then we have 
    \begin{align*}
        & \ax{\mc{D}_{00}}\noLine
            \uinf{\vdash\Gamma,A(x_1\ldots x_n)}
                \uinf{\vdash\Gamma,\sat{n}{n}(\lambda \vec{x} A,\vec{x})}
                \disp
         & \ax{\mc{D}_{10}}\noLine
            \uinf{\vdash\Delta,\ovl{A}(x_1\ldots x_n)}
                \uinf{\vdash\Delta,\ovl{\sat{n}{n}}(\lambda \vec{x} A,\vec{x})}
                \disp
    \end{align*}
    We can then simply apply the induction hypothesis to $\mc{D}_{00}$ and $\mc{D}_{10}$. If the last rules applied are $(\otimes)$ and $(\parr)$, respectively, we have:
    \begin{align*}
        &\ax{\mc{D}_{00}}\noLine
            \uinf{\vdash\Gamma,A}
        \ax{\mc{D}_{01}}\noLine
            \uinf{\vdash\Delta,B}
                \binf{\vdash\Gamma,\Delta, A\otimes B}
            \disp
        &&
        \ax{\mc{D}_{10}}\noLine
            \uinf{\vdash\Theta,\ovl{A},\ovl{B}}
                \uinf{\vdash\Theta,\ovl{A}\parr\ovl{B}}
                \disp
    \end{align*}
    Then the desired $\mc{D}$ is obtained by applying the induction hypothesis to, e.g., $\mc{D}_{00}$ and $\mc{D}_{10}$, and then to the resulting derivation and $\mc{D}_{01}$. It's here that the definition of $h(\cdot)$ plays a role: if length was defined as the number of nodes in the maximal branch of the proof-tree, then the induction would not go through in this case as, potentially,  $h(\mc{D})> h(\mc{D}_0)+h(\mc{D}_1)$. 
\end{proof}

$\al$ is known to be decidable. A natural question is whether the addition of rules for semantic notions preserves decidability. The answer may depend, of course, on which rules are added. 
\begin{obse}
   The system featuring the rules $\uts{1}{m}$ for all $m$ is undecidable. Therefore, so is $\mathtt{UTS}$. 
\end{obse}
\begin{proof}
  The claim can be established by interpreting Gri\^sin class theory in $\uts{1}{m}$. The translation scheme is quite straightforward, since the parameters in the rules of $\uts{1}{m}$ yield a notational variant of set-membership. The translation $\tau \colon \mc{L}_\in \to \lt$ leaves literals unchanged, commutes with the propositional connectives and quantifiers, and is such that, for $A$ a formula with $m$ free variables,
    \[
        (t\in \lambda x\,A)^\tau := \sat{1}{m}(\lambda x\,A, t).
    \]
   In fact,  by reversing the translation it's easy to see that $\uts{1}{m}$ and Gri\v{s}in's set theory are definitionally equivalent or synonymous in the sense of \cite{vis06}. 
\end{proof} 
\begin{op}
    Is $\uts{1}{0}$ decidable?
\end{op}

Cantini shows that the addition of a $\mathtt{K4}$ modality to Gri\v{s}in set theory -- that is, a rule corresponding to the modal principle $4$ -- and a necessitation rule is strong enough to derive the L\"ob's principle $\Box (\Box A \ra A) \ra \Box A$. We strengthen Cantini's observation and show that the schema $\uts{1}{0}$ suffices for the task. In what follows, it will be convenient to refer to the canonical name $\corn{A}$ of a sentence $A$ of $\mc{L}$, and to the corresponding truth-ascription $\T\corn{A}$. We let, for $A$ a sentence:
    \begin{align*}
        & \corn{A}:= \lambda v_0\,A, && \T\corn{A}:\lra \sat{1}{0}(\corn{A},v_0). 
    \end{align*}
\begin{dfn}
    The system $\uts{1}{0}+\mathtt{K4}$ is obtained by extending $\uts{1}{0}$ with the rules:
    \begin{align*}
        &\ax{\vdash\Diamond \Gamma,\Gamma, A}\Rlb{{\sc nec}}
            \uinf{\vdash\Delta,\Diamond \Gamma,\Box A}
                \disp
            &&\ax{\vdash\Diamond \Gamma, \Delta,A}
                \ax{\vdash\Diamond \Gamma,\Theta, B}\Rlb{$\Box \otimes$}
                    \binf{\vdash\Diamond\Gamma,\Delta,\Theta,A \otimes B}
                    \disp
                    &&&\ax{\vdash\Gamma,A}
                        \ax{\vdash\Delta, \ovl{A}}
                        \Rlb{\scriptsize Cut}
                        \binf{\vdash\Gamma,\Delta}
                        \disp
    \end{align*}
\end{dfn}
\begin{lemma}
   $ \uts{1}{0}+\mathtt{K4}$ derives the schema $\Box (\Box A\ra A)\ra \Box A$.
\end{lemma}
\begin{proof}
Let $C:\leftrightarrow (\Box\T\corn{C}\to A)$, for arbitrary $A$. We show that if $\Diamond \ovl{A} ,A$ is derivable, then so is $A$. We proceed as follows:
\smallskip
\begin{align*}
\scalebox{0.85}{
\ax{\vdash  \Diamond\ovl{\T}\corn{C},\Box\T\corn{C}}
\ax{\vdash \ovl{A},A}
\Rlb{$\otimes$}
\binf{ \vdash\Diamond\ovl{\T}\corn{C},\Box\T\corn{C}\otimes \ovl{A}, A}
\Rlb{$\ovl{\T}$}
\uinf{\vdash\Diamond\ovl{\T}\corn{C},\ovl{\T}\corn{C}, A }
\Rlb{$\Box$}
\uinf{\vdash \Diamond\ovl{\T}\corn{C}, \Box A }
\ax{\vdash \Diamond\ovl{A},A  }
\Rlb{\scriptsize Cut}
\binf{\vdash \Diamond\ovl{\T}\corn{C}, A}
\Rlb{$\parr$}
\uinf{\vdash \Diamond\ovl{\T}\corn{C}\parr A}
\ax{\vdash \Box\T\corn{C}\otimes \ovl{A}, C}
\Rlb{\scriptsize Cut}
\binf{\vdash C}
\Rlb{$\T$}
\uinf{\vdash\T\corn{C}}
\Rlb{$\Box$}
\uinf{\vdash \Box\T\corn{C}}
\ax{\vdash \Diamond\ovl{\T}\corn{C}, A}
\Rlb{\scriptsize Cut}
\binf{\vdash A}
\disp}
\end{align*}
From this point on, modal reasoning in G\"{o}del-L\"{o}b's provability logic as reported also in \cite[Thm 2.8]{can03} suffices.
\end{proof} 
By translating the box modality as $P\parr \ovl{P}$ for a designated atom $P$, we immediately obtain the conservativity of $ \uts{1}{0}+\mathtt{K4}$ over $\uts{1}{0}$ which in turns immediately yields the consistency of the former system.

We would like to conclude this section by observing that the calculus $\uts{1}{0}+\mathtt{K4}$ provably does not admit cut-elimination. To witness this it is enough to consider the sequent $\vdash\Diamond (\Box P\otimes \ovl{P}), \Box P$. The latter is indeed provable via cut as shown by the above derivation, but does not admit a cut-free proof by inspection of the rules.

\begin{op}
Can we obtain a cut-free system equivalent to $\uts{1}{0}+\mathtt{K4}$? A natural approach would be to substitute the modal rule with:
\begin{align*}
\ax{\vdash\Diamond\Gamma,\Gamma, \Diamond\ovl{A},A}
\uinf{\vdash\Diamond\Gamma,\Box A}
\disp
\end{align*}
\end{op}

The systems considered so far feature only additive quantifiers, which can be viewed as straightforward generalizations of the additive conjunction and disjunction. However, this straightforward solution to the logical paradoxes may not be completely satisfactory: the system lacks quantifiers that generalize multiplicative connectives. Several logicians and philosophers encouraged such a strengthening of the basic non-contractive theory \cite{bla92,mon04,mapa14}. The challenge was taken up by Zardini in \cite{zar11}.

\section{Multiplicative quantifiers and inconsistency}

\cite{zar11} attempts to establish a cut elimination theorem for an extension for the multiplicative fragment of affine logic extended with a combination of multiplicative quantifiers and na\"ive truth ($\ikto$). By $\biguplus_{i\in I}\Gamma_i$ we denote the infinitary multiset union of the $\Gamma_i$. Terms $t_1,t_2,t_3,...$ constitutes an exhaustive enumeration of the terms of the language.

\begin{align*}
         &\ax{}\Rlb{{\sc in}}
         \uinf{\vdash\Gamma,P,\ovl{P}}
            \DisplayProof   \\[10pt]
            &\ax{\vdash\Gamma,A }\Rlb{\T}
                \uinf{\vdash\Gamma,\T(\corn{A})}
                    \disp
                &&\ax{\vdash\Gamma,\ovl{A}}
                    \Rlb{$\ovl{\T}$}
                        \uinf{\vdash\Gamma,\ovl{\T}(\corn{A})}
                            \disp\\[10pt]
            &\ax{\vdash\Gamma,A,B}\Rlb{$\parr$}
                    \uinf{\vdash\Gamma,A\parr B}
                        \disp
                &&\ax{\vdash\Gamma,A}
                    \ax{\vdash\Delta,B}\Rlb{$\otimes$}
                        \binf{\vdash\Gamma,\Delta,A\otimes B}
                            \disp\\[10pt]
            & \ax{\hdots}
            \ax{\vdash\Gamma_i,A(t_i/x)}
            \ax{\hdots}\Rlb{$\forall$}
                \trinf{\vdash\biguplus_{i\in I}\Gamma_i,\forall x A}
                    \disp
                &&\ax{\vdash\Gamma,A(t_1/x),A(t_2/x),\hdots}\Rlb{$\exists$}
                    \uinf{\vdash\Gamma,\exists x A}
                        \disp
    \end{align*}

\smallskip

Zardini motivates the theory by emphasizing that additive connectives are not compatible with the solutions to the semantic paradoxes he defends; as a consequence, multiplicative quantifiers become the natural extension of multiplicative conjunction and disjunction. The proposal consists in equating multiplicative universal and existential quantifiers with an infinitary multiplicative conjunction and disjunction, respectively. This move is not without consequences from the point of view of the structural analysis of the system. In particular, the choice of such a reading of quantifiers has the immediate consequence of working with sequents with infinite multisets of formulas.


Several problems have been found with Zardini's proposal, but his work contains insightful ideas that prompted interest in the study of infinitary systems with multiplicative quantifiers and their interaction with paradox-breeding notions.  \cite{daro18} show that the extension of Zardini's system with basic arithmetical axioms lead to inconsistency. Moreover, \cite{fje20} isolates a gap in the cut-elimination proof. In \S\ref{sec:zrce}, we directly show that Zardini's cut-elimination algorithm is based on a proof-manipulation that does not preserve provability. In a recent paper \cite{fjol21}, it is shown that the system $\ikto$ is outright inconsistent, if the rules for the multiplicative quantifiers are used in a natural way to deal with vacuous quantification.  
%
%
In this section we show how, even without a truth predicate or similar semantic resources, the implicit rules for vacuous quantification in $\ikto$ are problematic. In particular, we prove that vacuous quantification simulates the role played by exponentials in linear logic. Therefore, vacuous quantification in the setting of Zardini's system allows one to faithfully interpret classical logic as a fragment. 

Since the system $\mathtt{IKT}_{\omega}$ and its fragment are systems in which derivations are infintely branching well-founded trees, we need to suitably modify the notion of height in order to carry out inductive arguments. To deal with infinitary derivations we assign ordinals to measure the heights of the derivations. The assignment is the standard one as can be found in \cite{sch77}, the key point is that for every rule $\rho$:
\begin{align*}
    \ax{\hdots}
    \ax{\Gamma_i}
    \ax{\hdots}
    \Rlb{$\rho$}
    \trinf{\Gamma}
    \disp
\end{align*}
the height of the premise $\Gamma_i$ is strictly less than the height of the conclusion $\Gamma$ for every $i$.
\subsection{Vacuous quantifiers and classical logic}

We start by showing that the rule for the existential quantifier is height-preserving invertible.
\begin{lemma}
The rule $\exists$ is height-preserving invertible.
\end{lemma}
\begin{proof}
By induction on the height of the derivation. If the sequent $\Gamma,\exists x A$ is an initial sequent, then so is $\Gamma, A(t_0/x), A(t_1/x),...$ If the formula $\exists x A$ is principal, the premise gives the desired conclusion. If the last rule applied is any other rule, we apply the induction hypothesis to each of the premise(s) and then the rule again. For example, if the last rule applied is $\forall$, we have:
\begin{align*}
\ax{\hdots}
\ax{\vdash\Gamma_i, B(t_i/y),\exists x A}
\ax{\hdots}
\Rlb{$\forall$}
\TrinaryInfC{$\vdash \Gamma, \forall y B,\exists x A$}
\disp
\end{align*}
We construct the following derivation:
\begin{align*}
\ax{\hdots}
\ax{\vdots \text{\scriptsize $\mathrm{IH}$}}
\noLine
\uinf{\vdash\Gamma_i, B(t_i/y),A(t_1/x),A(t_2/x),...}
\ax{\hdots}
\Rlb{$\forall$}
\TrinaryInfC{$\vdash \Gamma, \forall y B,A(t_1/x),A(t_2/x),...$}
\disp
\end{align*}
where $IH$ denotes the application of the inductive hypothesis.
\end{proof}
We also observe that the weakening rule (Weak) is height-preserving admissible in the system $\mathtt{IZ}_{\omega}$.
\begin{dfn}
The translation from classical logic in a language containing signed propositional atoms, conjunctions and disjunctions (in what follows we assume that the quantifiers are vacuous).
\begin{itemize}
\item $(P)^{*}=P$
\item $(\ovl{P})^{*}= \ovl{P}$
\item $(A\lor B)^{*}=\exists x A^{*}\parr \exists y B^{*}$
\item $(A\land B)^{*}= \exists x A^{*}\otimes \exists y B^{*}$
\end{itemize}
\smallskip
The translation extends to sequents: if $\Gamma$ is a finite multiset of formulae in the classical language, we let  $\Gamma^*=\exists x \Gamma^{*}$, where $\exists x \Gamma^*$ stands for the multiset obtained by prefixing every formula in $\Gamma$ with a vacuous quantifier. We write $A^\infty$ to denote the multiset of formula containing infinitely many copies of $A$. The definition naturally extends to multisets of formulas.
\end{dfn}

\begin{dfn}[$\mathtt{ALV}$]
    $\mathtt{ALV}$ extends $\mathtt{AL}$ with the following rules for vacuous quantification:
    \begin{align*}
        &
        \ax{\vdash\Gamma,A^\infty}
            \Rlb{v$\exists$}
                \uinf{\vdash\Gamma,\exists x A}
                    \disp
            && 
            \ax{\ldots}
            \ax{\vdash\Gamma_i,A}
            \ax{\ldots}\Rlb{v$\forall$}
                \trinf{\vdash\Gamma ,\forall x A}
                \disp
    \end{align*}
\end{dfn}

\smallskip

\begin{prop}
Classical propositional logic is a subsystem of affine propositional logic extended with infinitary rules for vacuous quantification ($\mathtt{ALV}$). 
\end{prop}

The proof of the proposition rests on the following Lemma which ensures the admissibility of an infinitary form of contraction for vacuously existentially quantified formulas.
\begin{lemma}\label{lem:exinf}
    The following rule is admissible in $\mathtt{ALV}$:
    \[
        \ax{\vdash\Gamma,\exists x A^\infty}
            \uinf{\vdash\Gamma,\exists x A}
                \DisplayProof
    \]
\end{lemma}

\begin{proof}
We argue by induction on the height of the derivation. If $\vdash\Gamma,\exists x A^\infty$ is an initial sequent, so is $\vdash\Gamma,\exists x A$, because only literals can be principal. If one of the existential quantifiers is principal, we have:
\begin{align*}
\ax{\vdash\Gamma,A^\infty, \exists x A^\infty }
\Rlb{v$\exists$}
\uinf{\vdash \Gamma,\exists x A^\infty  }
\disp
\end{align*}
By applying the invertibility of the rule for the existential quantifier we get a derivation of $\vdash\Gamma,A^\infty$, because the countable union of a countable multiset of formulas is a countable multiset. The desired conclusion follows by an application of the rule $\exists$. 

If the last rule is a unary rule and $\exists x A$ is not principal, we apply the induction hypothesis to the premise and then the rule again. If the last rule applied is R$\otimes$, we have:
\smallskip
\begin{align*}
\ax{\vdash \Gamma, B, \exists x A}
\ax{\vdash \Delta, C, \exists x A}
\Rlb{$\otimes$}
\binf{\vdash \Gamma,\Delta, \exists x A^\infty}
\disp
\end{align*}
In this case we construct the following derivation:
\smallskip
\begin{align*}
\ax{\vdash \Gamma, B, \exists x A}
\Rlb{\scriptsize inv}
\uinf{\vdash \Gamma, B, A^\infty}
\ax{\vdash \Delta, C, \exists x A}
\Rlb{\scriptsize inv}
\uinf{\vdash \Gamma, C, A^\infty}
\Rlb{$\otimes$}
\binf{\vdash \Gamma,\Delta, B\otimes C, A^\infty}
\Rlb{v$\exists$}
\uinf{\vdash \Gamma,\Delta, B\otimes C, \exists x A}
\disp
\end{align*}
\end{proof}

\begin{figure}\begin{center}
    \begin{tabular}{ c  c}
    \hline\hline\\
         $ \vdash\Gamma,P,\ovl{P}$ \;\;(\textsc{cin}) &\\[1em]
         \ax{\vdash\Gamma,A}
            \ax{\vdash\Gamma,B}\Rlb{$\land$}
                \binf{\vdash\Gamma,A\land B}
                    \disp
            &\ax{\vdash\Gamma, A,B}\Rlb{$\vee$}
                \uinf{\vdash\Gamma,A\vee B}
                    \disp \\[1em]
            \hline\hline
    \end{tabular}\label{tait}\caption{$\mathtt{CPL}$}
\end{center}
\end{figure}

\begin{proof}[Proof of Proposition]
    We first prove that, for $\Gamma$ a finite sequent in the classical logical language, 
    \beq\label{cl:trans}
        \text{$\mathtt{CL}$ derives $\Gamma$ only if $\mathtt{ALV}$ derives $\exists x \Gamma^*$}
    \eeq
    where $\mathtt{CL}$ is a Tait-style formulation of classical logic -- cf.~Figure \ref{tait}. \eqref{cl:trans} is obtained by induction on the length of the proof $n$ in $\mathtt{CL}$, where length can be taken to be the number of nodes in the maximal path of the derivation tree. If $n=0$, we have the following derivation of $\vdash \exists x P,\exists x \ovl{P}$ in $\mathtt{AVL}$
    \[
        \ax{\vdash P^\infty,\ovl{P}^\infty}\doubleLine
        \Rlb{v$\exists$}
            \uinf{\vdash \exists x P^\infty, \exists x \ovl{P}^\infty}
                \disp
    \]
    For $n> 0$, we consider the two different cases of ($\land$) and ($\vee$). In the former case, we reason as follows:
    \smallskip
    \[
        \ax{\vdash\exists x\Gamma^*, \exists x A^*}
            \ax{\vdash\exists x\Gamma^*, \exists x B^*}
            \Rlb{$\otimes$}
                \binf{\vdash(\exists x\Gamma^*)^2,\exists x A^*\otimes \exists x B^*}\Rlb{{\scriptsize Lemma \ref{lem:exinf}}}
                    \uinf{\vdash\exists x\Gamma^*,\exists x A^*\otimes \exists x B^*}
                    \Rlb{Weak}
                        \uinf{\vdash\exists x\Gamma^*,(\exists x A^*\otimes \exists x B^*)^\infty}\Rlb{v$\exists$}
                            \uinf{\vdash\exists x\Gamma^*,\exists x(\exists x A^*\otimes \exists x B^*)}
                    \disp
    \]
    
    \smallskip
    
    In the latter, we consider the following proof in $\mathtt{AVL}$:
    \smallskip
    \[
        \ax{\vdash\exists x \Gamma^*,\exists x A^*,\exists x B^*}
           \Rlb{$\parr$}
           \uinf{\vdash\exists x \Gamma^*,\exists x A^* \parr\exists x B^*}
           \Rlb{Weak}
                \uinf{\vdash\exists x \Gamma^*,(\exists x A^* \parr\exists x B^*)^\infty}\Rlb{v$\exists$}
                    \uinf{\vdash\exists x \Gamma^*, \exists x (\exists x A^* \parr \exists x B^*)}
            \disp
    \]
\end{proof}

The translation should be extended to first-order classical logic (this also possibly isolates another cut-free fragment of Zardini's system).

\begin{lemma}
If $\vdash A_{1}^{*\infty},\hdots,A_{n}^{*\infty}$ is derivable in $\mathtt{ALV}$, then $\mathtt{CL}$ derives $\vdash A_{1},\hdots,A_{n} $.
\end{lemma}
\begin{proof}
The proof is by induction on the height of the derivation in $\mathtt{ALV}$. If $\vdash A_{1}^{*\infty},\hdots,A_{n}^{*\infty}$ is an initial sequent, then $\vdash A_{1},\hdots,A_{n} $ is an initial sequent in $\mathtt{CL}$. If $\vdash A_{1}^{*\infty},\hdots,A_{n}^{*\infty}$ is the conclusion of a logical rule we distinguish cases according to the last rule applied. If the last rule applied is $\otimes$ we have:
\begin{align*}
\ax{\vdash \exists x B^{*} , (\exists x B^{*}\otimes \exists x C^{*})^{\infty},\hdots,A_{n}^{*\infty}}
\ax{\vdash \exists x C^{*}, (\exists x B^{*}\otimes \exists x C^{*})^{\infty},\hdots,A_{n}^{*\infty}}
\Rlb{$\otimes$}
\binf{\vdash\exists x B^{*}\otimes \exists x C^{*}, (\exists x B^{*}\otimes \exists x C^{*})^{\infty},\hdots,A_{n}^{*\infty}}
\disp
\end{align*}
We proceed as follows:
\begin{align*}
\ax{\vdash\exists x B^{*} , (\exists x B^{*}\otimes \exists x C^{*})^{\infty},\hdots,A_{n}^{*\infty}}
\Rlb{\scriptsize inv}
\uinf{\vdash B^{*\infty} , (\exists x B^{*}\otimes \exists x C^{*})^{\infty},\hdots,A_{n}^{*\infty}}
\Rlb{\scriptsize IH}
\doubleLine\dashedLine
\uinf{\vdash B ,  B \land   C ,\hdots,A_{n} }
\ax{\vdash \exists x C^{*}, (\exists x B^{*}\otimes \exists x C^{*})^{\infty},\hdots,A_{n}^{*\infty}}
\Rlb{\scriptsize inv}
\uinf{\vdash C^{*\infty}, (\exists x B^{*}\otimes \exists x C^{*})^{\infty},\hdots,A_{n}^{*\infty}}
\doubleLine\dashedLine
\Rlb{\scriptsize IH}
\uinf{\vdash C ,  B \land   C ,\hdots,A_{n} }
\Rlb{$\land$}
\binf{\vdash  B\land C ,  B \land   C ,\hdots,A_{n}}
\doubleLine\dashedLine
\Rlb{\scriptsize C}
\uinf{\vdash B \land   C ,\hdots,A_{n}}
\disp
\end{align*}
where $(C)$ denotes an application of height-preserving admissibility of the rule of contraction in the calculus for classical logic. 
If the last rule applied is $\parr$, we have:
\begin{align*}
    \ax{\vdash\exists x B^{*},\exists x C^{*}, (\exists x B\parr\exists x C)^{*\infty},\hdots, A_{n}^{*\infty}}
    \Rlb{$\parr$}
    \uinf{\vdash (\exists x B\parr\exists x C)^{*\infty},\hdots, A_{n}^{*\infty}}
    \disp
\end{align*}
We construct the following derivation:
\begin{align*}
    \ax{\vdash\exists x B^{*},\exists x C^{* }, (\exists x B\parr\exists x C)^{*\infty},\hdots, A_{n}^{*\infty}}
    \Rlb{\scriptsize inv}
    \doubleLine\dashedLine
    \uinf{\vdash   B^{*\infty},  C^{*\infty}, (\exists x B\parr\exists x C)^{*\infty},\hdots, A_{n}^{*\infty}}
    \Rlb{\scriptsize IH}
    \doubleLine\dashedLine
    \uinf{\vdash   B ,  C , B\lor C ,\hdots, A_{n} }
    \Rlb{$\lor$}
    \uinf{\vdash   B \lor C , B\lor C ,\hdots, A_{n} }
    \Rlb{\scriptsize C}
    \doubleLine\dashedLine
    \uinf{\vdash B\lor C ,\hdots, A_{n}}
    \disp
\end{align*}

\end{proof}
We can now prove the faithfulness of the embedding.
\begin{thm}
$\vdash\Gamma$ is derivable in $\mathtt{CL}$ if and only if $\vdash\exists x \Gamma^{*}$ is derivable in $\mathtt{ALV}$.
\end{thm}
\begin{proof}
From left to right we exploit the soundness of the translation. From right to left we apply invertibility of the rule for the existential quantifier and we get a derivation of $\vdash\Gamma^{*\infty}$. We then apply the faithfulness lemma which yields the desired conclusion.
\end{proof}

\subsection{Extension to first-order and infinitary logic}

We now extend to first-order logic the soundness of the embedding. To do so, we need to introduce clauses which translate the universal and the existential quantifiers. We propose the following:
\begin{itemize}
    \item $(\exists x A)^{*}=\exists x \exists y A^{*}$, $y$ does not occur in $A$.
    \item $(\forall x A)^{*}=\forall x \exists y A^{*}$, $y$ does not occur in $A$.
\end{itemize}
We recall the rules for the universal and existential quantifiers in classical logic.
\begin{figure}\begin{center}
    \begin{tabular}{ c  c}
    \hline\hline\\
            \ax{\vdash\Gamma,\exists x A, A(t/x)}\Rlb{$\exists$}
                \uinf{\vdash\Gamma,\exists x A}
                    \disp
            &\ax{\vdash\Gamma, A(y/x)}\Rlb{$\forall$, y!}
                \uinf{\vdash\Gamma,\forall x A}
                    \disp \\[1em]
            \hline\hline
    \end{tabular}\label{tait}\caption{Classical rules for quantifiers}
\end{center}
\end{figure}
\begin{prop}
The embedding extends to first-order classical logic. 
\end{prop}
\begin{proof}
We only need to check the case of the existential quantifier and the universal one. If the last rule applied is $\exists$, we have:
\begin{align*}
    \ax{\vdash\Gamma,\exists x A, A(t/x)}
    \Rlb{$\exists$}
    \uinf{\vdash\Gamma,\exists x A}
    \disp
\end{align*}
By induction on the height of the derivation we get:
\begin{align*}
    \ax{\vdash\exists y\Gamma,\exists y\exists x \exists y A^{*},  \exists y A^{*}(t/x)}
    \Rlb{\scriptsize Weak}
    \uinf{\vdash\exists y\Gamma,\exists y\exists x \exists y A^{*},  (\exists y A^{*}(t/x))^{\infty}}
    \Rlb{v$\exists$}
    \uinf{\vdash\exists y\Gamma,\exists y\exists x \exists y A^{*},  \exists x\exists y A^{*}(t/x) }
     \Rlb{\scriptsize Weak}
     \uinf{\vdash\exists y\Gamma,\exists y\exists x \exists y A^{*},  (\exists x\exists y A^{*}(t/x))^{\infty} }
      \Rlb{v$\exists$}
     \uinf{\vdash\exists y\Gamma,\exists y\exists x \exists y A^{*}}
     \disp
\end{align*}
In the case of the rule $\forall$, we proceed as follows:
\begin{align*}
    \ax{\hdots}
    \ax{\vdash\exists y \Gamma, \exists y A^{*}(t_i/x)}
    \ax{\hdots }
    \Rlb{v$\forall$}
    \trinf{\vdash(\exists y \Gamma)^{\infty}, \forall x\exists y A^{*}}
    \Rlb{Lemma \ref{lem:exinf}}
    \uinf{\vdash\exists y \Gamma , \forall x\exists y A^{*} }
    \Rlb{\scriptsize Weak}
    \uinf{\vdash\exists y \Gamma , (\forall x\exists y A^{*})^{\infty} }
    \Rlb{v$\exists$}
    \uinf{\vdash\exists y \Gamma , \exists y\forall x\exists y A^{*} }
    \disp
\end{align*}

\end{proof}

The embedding can be further extended to encompass infinitary classical logic, that is the extension of classical logic with the rule:
\smallskip
\[
    \ax{\vdash \Gamma,A(t_1/v)\ldots \vdash \Gamma,A(t_n/v)\ldots}\Rlb{$\forall^\infty$-cl}
        \uinf{\vdash \Gamma,\forall v A}
            \disp
\]
with $\Gamma$ a finite multiset. The claim follows immediately from
\begin{lemma}
    The rule \textsc{($\forall^\infty$-cl)} is admissible in $\mathtt{ALV}$ via the translation ${}^*$ of its formulas. 
\end{lemma}
\begin{proof}
We proceed as follows:
\begin{align*}
    \ax{\vdash\exists x\Gamma, \exists x A(t_1/v)}
    \ax{\hdots}
    \ax{\vdash\exists x\Gamma, \exists x A(t_n/v)...}
    \Rlb{v$\forall$}
    \trinf{\vdash(\exists x\Gamma)^{\infty}, \forall y\exists x A}
    \Rlb{Lemma \ref{lem:exinf}}
    \uinf{\vdash\exists x\Gamma, \forall y\exists x A}
    \Rlb{\scriptsize Weak}
    \uinf{\vdash\exists x\Gamma, (\forall y\exists x A)^{\infty}}
    \Rlb{v$\exists$}
    \uinf{\vdash\exists x\Gamma, \exists z\forall y\exists x A}
    \disp
\end{align*}
\end{proof}
In the case of infinitary classical logic, we can show that the embedding is indeed faithful, in the sense that if the translation of a sequent is provable in $\mathtt{ALV}$, then the sequent is provable in infinitary classical logic.
\begin{thm}
For any sequent $\vdash\Gamma$, if $\vdash\Gamma^{*\infty}$ is provable in $\mathtt{ALV}$, then $\vdash\Gamma$ is provable in infinitary classical logic.
\end{thm}
\begin{proof}
The proof is by induction on the height of the derivation in $\mathtt{ALV}$ distinguishing cases according to the last rule applied.

Suppose the last rule applied is $\forall$ with principal formula $\forall x \exists y A^{*}$, we have:
\smallskip
\begin{align*}
    \ax{\vdash\Gamma^{*\infty}, (\forall x \exists y A^{*})^{\infty},\exists y A^{*}(t_1/x)}
    \ax{\hdots}
    \ax{\vdash\Gamma^{*\infty},(\forall x \exists y A^{*})^{\infty}, \exists y A^{*}(t_n/x)...}
   \Rlb{$\forall^\infty$-cl}
    \trinf{\vdash\Gamma^{*\infty},(\forall x \exists y A^{*})^{\infty},\forall x \exists y A^{*}} 
    \disp
\end{align*}
we safely assume that the premises contain infinitely many copies of each of the formulas. 
We construct the following derivation:
\begin{align*}
    \ax{\vdash\Gamma^{*\infty}, (\forall x \exists y A^{*})^{\infty},\exists y A^{*}(t_1/x)}
    \Rlb{\scriptsize inv}
    \uinf{\vdash\Gamma^{*\infty}, (\forall x \exists y A^{*})^{\infty},(A^{*}(t_1/x))^{\infty}}
    \Rlb{\scriptsize IH}
    \doubleLine\dashedLine
    \uinf{\vdash \Gamma ,  \forall x A,A(t_1/x) }
    \ax{\hdots}
    \ax{\vdash\Gamma^{*\infty},(\forall x \exists y A^{*})^{\infty}, \exists y A^{*}(t_n/x)...}
     \Rlb{\scriptsize inv}
    \uinf{ \vdash\Gamma^{*\infty},(\forall x \exists y A^{*})^{\infty},(A^{*}(t_n/x))^{\infty}...}
    \Rlb{\scriptsize IH}
    \doubleLine\dashedLine
    \uinf{\vdash\Gamma ,\forall x  A,A(t_n/x) ... }
    \Rlb{$\forall^\infty$-cl}
    \trinf{\vdash\Gamma,\forall x A,\forall x A}
    \doubleLine\dashedLine
    \Rlb{\scriptsize C}
    \uinf{\vdash\Gamma,\forall x A}
    \disp
\end{align*}

\end{proof}
\subsection{Vacuous Quantification and Exponentials}
In this section we show that 
affine logic with exponentials, which in turn can be embedded via a faithful translation in $\mathtt{ALV}$.\footnote{It is fairly obvious that $\mathtt{ALV}$ can be faithfully translated in the extension of $\mathtt{AL}$ with infinitary rules for quantifiers.}

First we recall the rules which govern the exponentials in affine logic 
\smallskip
\begin{align*}
    & \ax{\vdash\Gamma,?A,?A}
    \Rlb{?c}
    \uinf{\vdash\Gamma, ?A}
    \disp
    && \ax{\vdash\Gamma, A}
    \Rlb{?}
    \uinf{\vdash\Gamma,?A}
    \disp
    &&& \ax{\vdash ?\Gamma,A}
    \Rlb{!}
    \uinf{\vdash \Delta,?\Gamma,!A}
    \disp
\end{align*}
We call $\mathtt{ALE}$ the resulting system -- Affine Logic with Exponentials.

Consider the translation:
\begin{itemize}
\item $(P)^\circ=P$
\item $(\ovl{P})^\circ= \ovl{P}$
\item $(A\parr B)^\circ=A^\circ\parr B^\circ$
\item $(A\otimes B)^\circ= A^\circ\otimes B^\circ$
\item $(? A)^\circ=\exists x A^\circ$
\item $(!A)^\circ=\forall x A^\circ$
\end{itemize}
where the quantifiers are vacuous. 
\begin{prop}\label{prop:e<->v}
$\vdash\Gamma$ is provable in $\mathtt{ALE}$ if and only if $\vdash\Gamma^\circ$ is provable in $\mathtt{ALV}$.
\end{prop}
The proof of Proposition \ref{prop:e<->v} follows immediately from the the next lemmata.
\begin{lemma}\label{lem:bang}
The following rule is admissible in $\mathtt{ALV}$ for every finite multiset $\Gamma$:
\begin{align*}
        \ax{\vdash\exists y\Gamma, A}
            \uinf{\vdash\exists y\Gamma,\forall x A}
                \disp
    \end{align*}
\end{lemma}
\begin{proof}
The admissibility is proved with the following steps.
\begin{align*}
\ax{\hdots}
\ax{\vdash \exists y \Gamma,A}
\ax{\hdots}
\Rlb{v$\forall$}
\trinf{\vdash (\exists y \Gamma)^\infty, \forall x A}
\Rlb{Lm. \ref{lem:exinf}}
\uinf{\vdash\exists y\Gamma,\forall x A}
\disp
\end{align*}
\end{proof}
\begin{lemma}
If $\mathtt{ALE}$ proves $\vdash\Gamma$, then $\mathtt{ALV}$ proves $\vdash\Gamma^\circ$.
\end{lemma}
\begin{proof}
We argue by induction on the height of the derivation of $\vdash\Gamma$ in $\mathtt{ALE}$. The only cases to check are the ones involving exponentials. If the last rule applied is $?$c or $!$ we exploit Lemma \ref{lem:exinf} and Lemma \ref{lem:bang}. If the last rule applied is $?$ we use height-preserving admissibility of weakening and the rule $\exists$.

\end{proof}
\begin{lemma}\label{lem:faith}
Let $\Gamma$ be a finite multiset of formulas of $\mathtt{ALE}$ and $A_1,...,A_n$ be formulas of $\mathtt{ALE}$:
\begin{center}
If $\mathtt{ALV}$ derives $\vdash\Gamma^\circ, A^{\circ\infty}_1,...,A^{\circ\infty}_n$, then $\vdash\Gamma, ?A_1,...,?A_n$ is derivable in $\mathtt{ALE}$.
\end{center}
\end{lemma}
\begin{proof}
We argue by induction on the height of the derivation of $\vdash\Gamma^\circ, A^{\circ\infty}_1,...,A^{\circ\infty}_n$ in $\mathtt{ALV}$ distinguishing cases according to the last rule applied. 

Since we are working in a setting with admissible weakening, we can safely assume that in applications of the rule $\otimes$ and $\forall$ for every $i\in\{1,\hdots, n\}$ infinitely many occurrences of $A^{\circ\infty}_i$ are present in each premise. 
If the last rule applied is $\forall$ and the principal formula is in $\Gamma^{\circ}$, we have:
\begin{align*}
    \ax{\hdots}
    \ax{\vdash\Gamma^{\circ'}_i,B^{\circ}  ,A^{\circ\infty}_1,...,A^{\circ\infty}_n}
    \ax{\hdots}
    \Rlb{v$\forall$}
    \trinf{\vdash\Gamma^{\circ'}, \forall x B^{\circ},A^{\circ\infty}_1,...,A^{\circ\infty}_n}
    \disp
\end{align*}
Since by assumption $\Gamma^{\circ'}$ is finite, there must be an $i<\omega$ such that $\Gamma_i=\emptyset$. We consider that premise $\vdash B^{\circ}  ,A^{\circ\infty}_1,...,A^{\circ\infty}_n$ and we construct the following derivation:
\begin{align*}
    \ax{\vdash B^{\circ}  ,A^{\circ\infty}_1,...,A^{\circ\infty}_n}
    \doubleLine\dashedLine
    \Rlb{\scriptsize IH}
    \uinf{\vdash B   ,?A_1,...,?A_n }
    \Rlb{!}
    \uinf{\vdash !B,?A_1,...,?A_n}
    \Rlb{Weak}
    \uinf{\vdash \Gamma',!B,?A_1,...,?A_n}
    \disp
\end{align*}
If $\forall x B$ is a formula among $A^{\circ\infty}_1,...,A^{\circ\infty}_n$ we proceed analogously with an extra application of the rule $?$. 

If the last rule applied is $\exists$ and the principal formula is among the formulas in $A^{\circ\infty}_1,...,A^{\circ\infty}_n$, we have:
\begin{align*}
    \ax{\vdash\Gamma^{\circ'},   B^{\circ\infty},A^{\circ\infty}_1,...,A^{\circ\infty}_n}
    \Rlb{v$\exists$}
    \uinf{\vdash\Gamma^{\circ'}, \exists x B^{\circ},A^{\circ\infty}_1,...,A^{\circ\infty}_n}
    \disp
\end{align*}
We construct the following derivation:
\begin{align*}
    \ax{\vdash\Gamma^{\circ'},   B^{\circ\infty},A^{\circ\infty}_1,...,A^{\circ\infty}_n}
    \Rlb{\scriptsize IH}
    \doubleLine\dashedLine
    \uinf{\vdash\Gamma' ,   ?B ,?A_1,...,?A_n }
    \disp
\end{align*}
The application of the inductive hypothesis suffices.

The remaining cases are easily provable by applications of the inductive hypothesis followed by applications of the rules of the calculus $\mathtt{ALE}$.
\end{proof}
Lemma \ref{lem:faith} gives a formal representation of the intuitive claim about the infinitary nature of exponentials. Indeed, the context-restriction imposed on the rule for the operator $!$ is simulated by the fact that the infinitary multiplicative rule for $\forall$ yields a premise in which the context not under the scope of $?$ is absent.
\begin{remark}
We observe that due to the transitivity of faithful translations we obtain an alternative proof of the embedding of classical logic into $\mathtt{ALV}$ as follows:
\begin{center}
$\mathtt{CL}$ proves $\vdash\Gamma \Leftrightarrow $ $\mathtt{ALE}$ proves $\vdash\Gamma^\bullet\Leftrightarrow$ $\mathtt{ALV}$ proves $\vdash (\Gamma^\bullet)^\circ$ 
\end{center}
where $\bullet$ is the translation of affine logic into classical logic.
\end{remark}
\subsection{Exponential Liar} From the previous results linking vacuous quantification and the exponentials, and the inconsistency Zardini's system established by \cite{fjol21}, we can restore the propositional structure of the derivation of the Liar paradox in full linear and affine logics extended with rules for full disquotation. 
%
{By our assumptions on $\lambda$-terms, we can assume that there is a term $l:=\corn{?\ovl{\T}(l)}$. We abbreviate with $L$ the sentence $?\ovl{\T}(l)$. We are also assuming that $\ovl{L}$ abbreviates $!\T (l)$. Therefore, the rules
\begin{align*}
    &\ax{\vdash\Gamma, ?\ovl{\T}(l)}
    \Rlb{$L$}
    \uinf{\vdash\Gamma, L}
    \disp
    && \ax{\vdash\Gamma, !\T(l)}
    \Rlb{$\ovl{L}$}
    \uinf{ \vdash\Gamma, \ovl{L}}
    \disp
\end{align*}
are obviously admissible -- in fact, the conclusions are just notational variants of the premisses.}
\begin{prop}\label{prop:liar}
Full, propositional linear and affine logics are inconsistent with the rules
\begin{align*}
    &\ax{\vdash\Gamma,A}\Rlb{$\T$}
        \uinf{\vdash\Gamma,\T\corn{A}}
            \DisplayProof
        &&
            \ax{\vdash\Gamma,\ovl{A}}\Rlb{$\ovl{\T}$}
                \uinf{\vdash\Gamma,\ovl{\T}\corn{A}}
                    \DisplayProof
\end{align*}
for $A$ a sentence possibly containing exponentials.
\end{prop}
\begin{proof}
\begin{align*}
\ax{}
\Rlb{in}
\uinf{\vdash\ovl{\T}(l),\T(l)}
\Rlb{?}
\uinf{\vdash?\ovl{\T}(l),\T(l)}
\Rlb{!}
\uinf{\vdash?\ovl{\T}(l),!\T(l)} 
\Rlb{$\ovl{L}$}
\uinf{\vdash?\ovl{\T}(l),\ovl{L}}
\Rlb{$\ovl{\T}$}
 \uinf{\vdash ?\ovl{\T}(l),\ovl{\T}(l)}     \Rlb{?}
   \uinf{\vdash ?\ovl{\T}(l),?\ovl{\T}(l)}     
   \Rlb{?c}
   \uinf{\vdash ?\ovl{\T}(l)}   
   \Rlb{$L$}
   \uinf{\vdash L}  
  \ax{}
\Rlb{in}
\uinf{\vdash\ovl{\T}(l),\T(l)}
\Rlb{?}
\uinf{\vdash?\ovl{\T}(l),\T(l)}
\Rlb{!}
\uinf{\vdash?\ovl{\T}(l),!\T(l)} 
\Rlb{$\ovl{L}$}
\uinf{\vdash?\ovl{\T}(l),\ovl{L}}
\Rlb{$\ovl{\T}$}
 \uinf{\vdash ?\ovl{\T}(l),\ovl{\T}(l)}     \Rlb{?}
   \uinf{\vdash ?\ovl{\T}(l),?\ovl{\T}(l)}     
   \Rlb{?c}
   \uinf{\vdash ?\ovl{\T}(l)}   
   \Rlb{$L$}
   \uinf{\vdash L}   
   \Rlb{\T}
   \uinf{\vdash \T(l)}
   \Rlb{!}
   \uinf{\vdash !\T(l)}
   \Rlb{$\ovl{L}$}
    \uinf{\vdash \ovl{L}}
    \Rlb{cut}
   \binf{\vdash }
   \disp
   \end{align*} 
\end{proof}
\begin{remark}
The content of Proposition \ref{prop:liar} shows that - in general - full linear logic with exponentials is enough to simulate the liar paradox when paired with rules for na\"ive truth. We would like to point out that in our setting the faithful embedding of the exponentials in $\mathtt{ALV}$ requires the presence of the structural rule of weakening.
\end{remark}
\section{Cut-elimination for multiplicative quantifiers}

\subsection{Zardini's cut-elimination: another visit}\label{sec:zrce}
The results in the previous sections tell us that Zardini's cut-elimination argument for the theory of na\"ive truth based on his multiplicative quantifiers cannot work. This leaves open the question whether Zardini's procedure could work in the absence of the rules for the truth predicate. The answer is still negative: \cite{fje20} found a gap in Zardini's reduction for the quantifiers.  Fjellstad isolates an example of a sequent which is obviously cut-free derivable, but such that the cut involved in its proof cannot be eliminated following Zardini's instructions. Although pointing to a serious gap in Zardini's reduction, Fiellstad's example involves a case that can nonetheless be dealt with by supplementing Zardini's original reduction strategy with extra conditions.\footnote{To be sure, we believe that Fjellstad's example points to a fundamental flaw in Zardini's strategy, but the specific example does not amount to a knock-down case.} By contrast, we directly show that Zardini's cut-elimination algorithm is based on a proof-manipulation that does not preserve provability. 

The problem involves the elimination of cuts in which the cut formula is principal in both the premises of the cut and is a universal or existential formula. Consider the cut which needs to be eliminated.
\begin{prooftree}
\AxiomC{$\hdots$}
 \AxiomC{$ \vdash\Gamma_{i}, ,A(t_{i}/x) $}
 \AxiomC{$\hdots$}
 \RightLabel{\scriptsize $\forall$}
    \TrinaryInfC{$ \vdash\Gamma, \forall x  A$}
    \AxiomC{$\vdash \overline{A}(t_{1}/x),\overline{A}(t_{2}/x),...,\Delta$}
    \RightLabel{\scriptsize $\exists$}
    \UnaryInfC{$\vdash\exists x  \overline{A},\Delta $}
    \RightLabel{\scriptsize Cut}
    \BinaryInfC{$\vdash\Gamma,\Delta $}
    
    \end{prooftree}
The solution proposed by Zardini is to reduce the size of the multiset of cut formulas $\overline{A}(t_{1}/x),\overline{A}(t_{2}/x),...$ introduced by the application of $\exists$. In particular, one should trace up the multiset in the derivation until it becomes finite in a branch. 
By the design of the system a countably infinite (sub)multiset of $\overline{A}(t_{1}/x),\overline{A}(t_{2}/x),...$ can only be introduced by the rule $\forall$ or by a weakened initial sequent, we detail the first case. 
\begin{prooftree}

\AxiomC{$\hdots$}
 \AxiomC{$ \vdash\Gamma_{i}, ,A(t_{i}/x) $}
 \AxiomC{$\hdots$}
 \RightLabel{\scriptsize $\forall$}
    \TrinaryInfC{$ \vdash\Gamma, \forall x  A$}
    \AxiomC{$\hdots$}
    \AxiomC{$   \vdash\overline{A}(t_{i}/x),\Delta''_{i}$}
    \AxiomC{$\hdots$}
     \RightLabel{\scriptsize $\forall$}
    \TrinaryInfC{$\vdash\overline{A}(t_{i}/x),\overline{A}(t_{i+1}/x), \Delta'  $}
    \noLine
    \UnaryInfC{$\vdots$ $\mathcal{D}$}
    \noLine
    \UnaryInfC{$\vdash \overline{A}(t_{1}/x),\overline{A}(t_{2}/x),..., \Delta$}
    \RightLabel{\scriptsize $\exists$}
    \UnaryInfC{$\vdash\exists x\ovl{A},\Delta $}
    \RightLabel{\scriptsize Cut}
    \BinaryInfC{$\vdash\Gamma,\Delta $}
    \end{prooftree}
    \smallskip
Notice that the principal formula in $\forall$ is not displayed.   According to Zardini, we should pick the premise $  \vdash\overline{A}(t_{i}/x),\Delta''_{i}$ and construct the following derivation.
\begin{prooftree}
    \AxiomC{$  \vdash\overline{A}(t_{i}/x),\Delta''_{i} $}
    \noLine
    \UnaryInfC{$\vdots$ $\mathcal{D}$}
    \noLine
    \UnaryInfC{$\vdash \ovl{A}(t_{i}/x),\ovl{A}(t_{i}/x),...,\ovl{A}(t_{i}/x), \Delta'  $}
    \end{prooftree}

\smallskip

The cut is then replaced by $i$ many cuts and the desired conclusion follows from the application of the weakening rule. Now, the gap in Zardini argument is exactly in the passage displayed above. In fact, while the sequent $  \vdash\overline{A}(t_{i}/x),\Delta''_{i}  $ is indeed provable, the same cannot be said of the sequent $   \vdash \ovl{A}(t_{1}/x),\ovl{A}(t_{2}/x),...,\ovl{A}(t_{i}/x), \Delta' $. {In other words, Zardini's reduction is based on the idea that the derivation $\mc{D}$ could be performed \emph{even if one focused on a single premiss only, instead of infinitely many.}} 
{For instance, according to the reduction, one could start with the derivation}
\begin{center}
\AxiomC{$\hdots$}
\AxiomC{$\vdash \ovl{P}(t_{i}/x) , P(t_{i}/x) $}
\AxiomC{$\hdots$}
     \RightLabel{\scriptsize $\forall$}
    \TrinaryInfC{$\vdash\ovl{P}(t_{i}/x),\ovl{P}(t_{i+1}/x),..., \forall x P $}
    \noLine
    \UnaryInfC{$\vdots$ $\mathcal{D}$}
    \noLine
    \UnaryInfC{$\vdash \ovl{P}(t_{1}/x),\ovl{P}(t_{2}/x),...,\Delta'$}
    \DisplayProof
    \end{center}
    
    \smallskip
{According to the reduction, one could then transform the derivation into:}
\begin{center}
\AxiomC{$\vdash \ovl{P}(t_{i}/x) , P(t_{i}/x) $}
     \RightLabel{\scriptsize $\forall$}
    \UnaryInfC{$\vdash\ovl{P}(t_{i}/x), \forall x P $}
    \noLine
    \UnaryInfC{$\vdots$ $\mathcal{D}$}
    \noLine
    \UnaryInfC{$\vdash\ovl{P}(t_{i}/x),\ovl{P}(t_{2}/x),...,\Delta'$}
    \DisplayProof
    \end{center}
The sequent $\vdash\ovl{P}(t_{i}/x), \forall x P  $, however, is clearly not (cut-free) provable. 

\subsection{Eliminating cuts}
Zardini's reduction is flawed even if one considers the system without the truth predicate. However, as we shall now demonstrate, cut is eliminable in Zardini's infinitary logic (without truth), i.e. the system $\mathtt{IK}_{\omega}$.
\begin{align*}
         &\ax{}\Rlb{{\sc in}}
         \uinf{\vdash\Gamma,P,\ovl{P}}
            \DisplayProof   \\[10pt]
            &\ax{\vdash\Gamma,A,B}\Rlb{$\parr$}
                    \uinf{\vdash\Gamma,A\parr B}
                        \disp
                &&\ax{\vdash\Gamma,A}
                    \ax{\vdash\Delta,B}\Rlb{$\otimes$}
                        \binf{\vdash\Gamma,\Delta,A\otimes B}
                            \disp\\[10pt]
            & \ax{\hdots}
            \ax{\vdash\Gamma_i,A(t_i/x)}
            \ax{\hdots}\Rlb{$\forall$}
                \trinf{\vdash\biguplus_{i\in I}\Gamma_i,\forall x A}
                    \disp
                &&\ax{\vdash\Gamma,A(t_1/x),A(t_2/x),\hdots}\Rlb{$\exists$}
                    \uinf{\vdash\Gamma,\exists x A}
                        \disp
    \end{align*}

\smallskip

Our strategy is based on a double induction, on the length of the derivation and on the grade of the cut formula: for this reason, the proof cannot be lifted to the system with a fully disquotational truth predicate since, as it is well-known, truth collapses the grade of sentences. 

We shall eliminate cuts of the form:
  
    \smallskip
\begin{align*}
    \ax{\vdash\Gamma, \Phi}
    \ax{\{\vdash\Delta_{\varphi},\overline{\varphi}\, |\, \varphi\in\Phi\}}
    \Rlb{Cut}
    \binf{\vdash\Gamma,\Delta}
    \disp
\end{align*}
Intuitively, the (CUT) rule allows to cut infinitely many formulas simultaneously. Hence we have one premise $\vdash\Gamma,\Phi$, where $\Phi$ is the multiset of formulas to cut and (possibly) infinitely many premises $\vdash\Delta_\varphi,\ovl{\varphi}$, one for every formula $\varphi\in\Phi$. Finally, the multiset $\Delta$ in the conclusion denotes the infinitary multiset union of all the multisets $\Delta_\varphi$.

    The depth of a formula $dp(\varphi)$ is the number of logical connectives (including quantifiers) occurring in it. We shall reason by double induction, with main induction hypothesis on the degree of the multiset of cut formulas, i.e. $dg(\Phi)=sup_{\varphi\in\Phi} (dp(\varphi))+1$ (the degree of a multiset will be - in general - an ordinal), and secondary induction hypothesis on the Hessenberg ordinal sum of the height of the derivations (which is commutative, associative, left and right cancellative and strictly monotone in both arguments). The key point of the reduction is the fact that infinite multisets of the form $[A(t_{i}/x)\,|\, i\in I]$ have a {\it finite degree}, because all the formulas occurring inside them have the same degree.
    
We first prove an auxiliary lemma which enables us to remove cuts on atomic formulas.
\begin{lemma}
For any multiset $\Gamma,\Delta$ and any literal $P$, the rule:
  
    \smallskip
\begin{align*}
    \ax{\vdash\Gamma,P}
    \ax{\vdash\Delta,\ovl{P}}
    \Rlb{Cutat}
    \binf{\vdash\Gamma,\Delta}
    \disp
\end{align*}
is admissible.
\end{lemma}
\begin{proof}
The proof is by induction on the height of $\vdash\Gamma, P$. If $\Gamma, P$ is an initial sequent, the proof follows by admissibility of weakening. If $\vdash\Gamma, P$ is not an initial sequent, then it is the conclusion of a rule and $P$ cannot be the principal formula. In this case, we permute the cut upward and we eliminate it by induction on the height of the derivation.
\end{proof}
\begin{thm}
The cut rule is admissible in $\mathtt{IK}_{\omega}$.
\end{thm}
\begin{proof}
By double (transfinite) induction with main induction hypothesis on the degree of the multiset of cut formulas and secondary induction hypothesis on the height of the left premise of the cut, i.e. $\Gamma,\Phi$.

If $\vdash\Gamma,\Phi$ is an initial sequent, we distinguish cases. If no formula is active in $\Phi$, then $\vdash\Gamma,\Delta$ is an initial sequent too. If one formula is active in $\Phi$, then the proof follows by weakening. If both the atomic formulas are active in $\Phi$, i.e. if $\Phi\equiv \Phi',P,\ovl{P}$, then we have two premises $\vdash\Delta_{P}, P$ and $\vdash\Delta_{\ovl{P}}, \ovl{P}$ and the desired conclusion follows by an application of the admissible rule \textsc{Cutat}. 

If no formula in $\Phi$ is principal, the cut is permuted upwards (possibly replaced by infinitely many cuts) and removed by secondary induction hypothesis. 

If a formula is principal in $\Phi$, we distinguish cases according to its shape. We focus on the cases of the quantifiers, as they are the relevant ones. If a formula of the shape $\forall x A$ is principal, we have:
\begin{align*}
    \ax{\vdash\Gamma_1,\Phi_1,A(t_1/x)}
    \ax{\hdots}
    \ax{\vdash\Gamma_n,\Phi_n,A(t_n/x)\hdots}
    \Rlb{$\forall$}
    \trinf{\vdash\Gamma,\Phi,\forall x A}
    \disp
\end{align*}
The other premises of the cut will be $\Delta,\exists x \overline{A}$ and $\Theta_{\varphi},\overline{\varphi}$ for every $\varphi$ in $\Phi$. First, for every $i\in I$, we perform the following reduction:
\begin{align*}
    \ax{\vdash\Gamma_i,\Phi_i,A(t_i/x)}
    \ax{\{\vdash\Theta_\varphi, \overline{\varphi}\, |\, \varphi\in\Phi_{i}  \}}
    \Rlb{Cut}
    \binf{\vdash\Gamma_i,\Theta_i, A(t_i/x)}
    \disp
\end{align*}
The cut is removed by secondary induction hypothesis on the height of the left premise of the cut. We then apply height-preserving invertibility of the rule $\exists$ to $\vdash\Delta,\exists x \overline{A}$ to get $\vdash\Delta, \overline{A}(t_1/x),\overline{A}(t_2/x),\hdots$. Finally we proceed with the following cut:
\begin{align*}
    \ax{\vdash\Delta, \overline{A}(t_1/x),\overline{A}(t_2/x),\hdots}
    \ax{\{\vdash\Gamma_i,\Theta_i, A(t_i/x)\, |\, i\in I\}}
     \Rlb{Cut}
    \binf{\vdash\Gamma,\Theta, \Delta}
    \disp
\end{align*}
This cut is removed by primary induction hypothesis on the degree of the multiset of cut formulas which is strictly decreased.

If the principal formula is an existential one, we have
\begin{align*}
    \ax{\vdash\Gamma,\Phi,A(t_1/x),A(t_2/x),\hdots}
     \Rlb{$\exists$}
    \uinf{\vdash\Gamma,\Phi,\exists x A}
    \disp
\end{align*}
In this case we look at the premise of the cut of the shape $\vdash\Delta,\forall x\overline{A}$ and we distinguish two subcases. Either $\forall x\overline{A}$ is principal in an inference rule in the derivation or not. In the latter case, then $\Delta$ is already derivable and we obtain the desired conclusion via weakening. In the former case we go upwards to the point in which $\forall x\overline{A}$ is principal (by the design of the rules $\forall x\overline{A}$ will be only in one branch). We have:
\begin{align*}
    \ax{\vdash\Delta'_1, \overline{A}(t_1/x)}
    \ax{\hdots}
    \ax{\vdash\Delta'_n,\overline{A}(t_n/x)}
     \Rlb{$\forall$}
    \trinf{\vdash\Delta',\forall x \overline{A}}
    \noLine
    \uinf{\vdots\pi}
    \noLine
    \uinf{\vdash\Delta,\forall x \overline{A}}
    \disp
\end{align*}
We perform the following reduction:
  
    \smallskip
\begin{align*}
    \ax{\vdash\Gamma,\Phi,A(t_1/x),A(t_2/x),\hdots}
    \ax{\{\vdash\Theta_{\varphi},\overline{\varphi}\,|\,\varphi\in\Phi\}}
    \Rlb{Cut}
    \binf{\vdash\Theta,\Gamma,A(t_1/x),A(t_2/x),\hdots}
     \ax{\{\vdash\Delta'_i, \overline{A}(t_i/x)\, |\,i\in I\}}
     \Rlb{Cut}
    \binf{\vdash\Theta,\Gamma,\Delta'}
    \noLine
    \uinf{\vdots\pi}
    \noLine
    \uinf{\vdash\Theta,\Gamma,\Delta}
    \disp
\end{align*}
The topmost cut is removed by secondary induction hypothesis on the height of the left premise of the cut, whereas the lowermost is removed by induction on the degree of the multiset of cut formulas which has - again - strictly decreased.
\end{proof}
We have introduced an approach to cut-elimination for multiplicative quantifiers. It seems hard to generalize it so as to encompass a theory of truth (we use a double induction and one of the measures is a kind of degree of formulas). However, this is coherent, as the original system by Zardini is inconsistent. We believe that -- as pointed out also in \cite{petersen2022} -- the explicit presence of a double inductive parameter in the cut-elimination procedure brings to the fore the hidden presence of contraction. 
\section{Concluding remarks and future work}
We investigated contraction-free systems and their applicability to the solution of paradoxes in the context of theories of truth. Furthermore, we proposed a new conceptualization of exponentials, thus giving an alternative interpretation of an intrinsically modal notion. Finally, the study led us to a proof-theoretical analysis of multiplicative quantifiers by means of a new cut-elimination procedure for infinitary sequents.

In conclusion, we would like to sketch some open problems which are worth addressing. To start with, it would be interesting to find a suitable truth predicate to add to the base theory while preserving consistency. The task is not trivial, because, as shown, systems based on multiplicative quantifiers are not entirely contraction-free.

Moreover, Grishin set theory is inconsistent modulo the addition of extensionality. A natural question arises as to whether there exists a natural corresponding property in the case of truth theories based on contraction-free systems with additive (or classical, one may say) quantifiers. 

Finally, it would be important to explore whether the cut-elimination theorem can be generalized to the case of infinitary logic with infinite sequents. In particular, it would be interesting to study the strength of the resulting system.
\bibliography{mult_bib}
\bibliographystyle{alpha}
\end{document}